\documentclass[]{amsart}
\usepackage{amsmath}
\usepackage{amsfonts}
\usepackage{amssymb}
\usepackage{amsthm}
\usepackage{enumitem}
\usepackage{graphicx,color,xcolor,tikz}
\usepackage[utf8]{inputenc}
\usepackage[notref,notcite]{showkeys}
\usepackage{hyperref}
\hypersetup{
    colorlinks=true,                         
    linkcolor=blue, 
    citecolor=red, 
  } 

\newtheorem{Lem}{Lemma}
\newtheorem{Thm}{Theorem}
\newtheorem{Prop}{Proposition}
\newtheorem{Def}{Definition}
\newtheorem{Rem}{Remark}
\newtheorem{Hyp}{Assumption}

\def\R{\mathbb R}

\newcommand{\p}{\partial}
\DeclareMathOperator*{\argmax}{argmax}

\title[A
  thermodyn. consistent model of a liquid-vapor fluid with a
  gas]{A
  thermodynamically consistent model of a liquid-vapor fluid with a
  gas}

\author{H\'el\`ene Mathis}

\address{Laboratoire de Math\'ematiques Jean Leray, Universit\'e de
  Nantes \& CNRS UMR 6629, 
  BP 92208, F-44322 Nantes Cedex 3, France}

\email{helene.mathis@univ-nantes.fr}

\begin{document}

\maketitle

\begin{abstract}
This work is devoted to the consistent modeling of a three-phase
mixture of a gas, a liquid and its vapor. Since the gas and the vapor
are miscible, the mixture is subjected to a non-symmetric constraint on
the volume. Adopting the Gibbs formalism, 
the study of the extensive equilibrium entropy of the
system allows to recover the Dalton's law between the two gaseous
phases. In addition, we distinguish whether phase transition occurs or
not between the liquid and its vapor. The thermodynamical equilibria
are described both in extensive and intensive variables. In the latter
case, we focus on the geometrical properties of equilibrium entropy.
The consistent characterization of the thermodynamics of the
three-phase mixture is used to introduce two Homogeneous Equilibrium
Models (HEM) depending on mass transfer is taking into
account or not. Hyperbolicity is investigated while analyzing the entropy
structure of the systems. 
Finally we propose two Homogeneous
Relaxation Models (HRM) for the three-phase mixtures with and without
phase transition. Supplementary equations on mass, volume and energy
fractions are
considered with appropriate source terms which model the relaxation
towards the thermodynamical equilibrium, in agreement with entropy
growth criterion.
\end{abstract}


\noindent
\textbf{2010 MCS.} 76T30, 80A10, 35Q79.

\tableofcontents

\section{Introduction}

The modelling of compressible multiphase flows is crucial for a wide
range of applications, notably in the nuclear framework, for instance
in vapor explosion of for fast transient situations \cite{bartak90, xie2013}.  
Within the two last decades, this topic has resulted in an abundant
literature especially about
two-phase flows, see for instance
\cite{BN86,kapila,GavrilyukSaurel02,Flatten,dreyer12,dreyer14}.
 More recently attention has been paid to the simulation of
 three-phase flows \cite{herard07, herard16, HantkeMuller16,
   HantkeMuller17}, by means of relaxation models in the spirit of the
 two-fluid Baer and Nunziato model \cite{BN86}. 
In all the latter references the mixture is assumed to be immiscible
that is all the phases occupy different volumes. The thermodynamical
equilibrium of the mixture is then depicted by the equality of the
pressures and temperatures of the three phases (and also chemical
potential as phase transition is considered).
As the mixture dynamics is considered, each phase dynamic is depicted
by an Euler type system which are coupled through non-conservative interfacial terms,
 additional advection equations of volume fractions and relaxation terms. 
The overall system enters the class of hyperbolic system of
relaxation and admits good properties: hyperbolicity, well-understood
wave structure, entropy inequality... For particular Equations of
State (EoS), the Riemann problem is also well understood and has
lead to the development of relevant numerical approximation (see again
\cite{herard07, herard16, HantkeMuller16,
   HantkeMuller17} for three-fluid
 (perfectly immiscible)
 models).
As immiscible mixture are considered, that is when the phases are
intimate and share the same volume, one should refer to the works of
Dellacherie \cite{Rency01, Dellacherie03}. 
Since the phases are miscible, the model is in adequacy with the expected Dalton's law
which states that the equilibrium mixture pressure is the sum of the
pressures of each phases. We refer to \cite{Gibbs} and \cite{callen85} for
fundamental Thermodynamics. 
The dynamic of the multicomponent fluid is again described by a Baer
and Nunziato type of system, including relaxation terms and
non-conservative interfacial terms. The authors also investigate the
impact of the closure law on the hyperbolicity of the associated
Homogeneous Equilibrium Model (HEM). These works
complement the study proposed in \cite{Lagoutiere00} about the
comparison of several closure laws applied to an HEM model in the case of a
multicomponent immiscible mixture.

The purpose of the present work is to investigate the thermodynamics
of a mixture which is not merely miscible or immiscible but of mix type.
We focus on a three-phase compressible flows,
composed of a liquid phase, its associated vapor phase and a
gas. 
The gas is miscible with the vapor phase but no mass transfer can
occur between either the gas and the vapor or the gas and the liquid.
Besides phase transition can occur between the vapor and the liquid;
in the whole paper we will distinguish whether phase transition between
the liquid and the vapor occur or not. The core of the paper is the
modelling of a rigorous thermodynamical model. It allows to construct
reliable hyperbolic HEM models to depict the motion of the compressible
three-phase mixture. We do not to address numerical aspects
because of lack of relevant test cases.

First we aim at precisely give an accurate description of the
thermodynamical equilibrium of the system. Adopting the Gibbs
formalism, as done in \cite{HS06, Mathis10,HM10}, we intricate the
extensive variables of the system.
This description relies on the definition of the extensive equilibrium
entropy of the system. The second law of Thermodynamics states that
the thermodynamical equilibrium is attained as the mixture entropy
reaches its maximum under some constraints. Depending on phase
transition occurs or not
between the liquid and its vapor, the set of constraints
changes, leading to different properties on the entropy function.
The core issue is the volume constraint which reflects the
non-symmetric immiscibility properties between the liquid and the
gaseous phases. This constraint makes the whole modeling difficult
since it prevents from using convenient tools of convex analysis
such as sub-convolution and Legendre transform, see \cite{Mathis10,HM10}.
At this stage, one recovers a consistent characterization of the
thermodynamical equilibrium : the Dalton's law for the gaseous
phases and the equality of the temperatures apply.
Note that a similar description (in terms of extensive variables) has
been proposed in \cite{BMHM12} but the computations are restricted to
particular equations of state (namely stiffened gas laws for the three
phases). The present study is valid for any equations of state.
Turning to the intensive variables, we analyse the specific
equilibrium entropies in terms of optimization problems in the spirit
of \cite{HS06, faccanoni07,Mathis10,HM10, faccanoni12}. 

Section \ref{sec:modell-assumpt}
addresses the construction of three-phase Euler systems at
thermodynamical equilibrium called HEM models.
Following the works of Dellacherie \cite{Rency01, Dellacherie03}, it
consists in providing the correct closure
laws to the three-phase Euler system in agreement with the
optimization constraints presented in Section \ref{sec:thermo}. We distinguish two cases
depending on phase transition occurs or not between the vapor and the
liquid phases. When phase transition is omitted, we prove that the
resulting system is hyperbolic using
a modified Godunov-Mock theorem in the spirit of
\cite{Lagoutiere00}. When mass transfer is allowed, hyperbolicity is
also proven. But the extension of the Godunov-Mock theorem is 
obsolete and one has to go back to the study the Jacobian matrix of the flux.

One difficulty when approximating solutions of HEM models is that the
mixture pressure is often difficult to express analytically, even when
the phases are depicted by simple EoS, see for instance the
computations detailled in \cite{BMHM12} for a three-phase mixture.
Besides the mixture pressure law may 
present pathologies leading to the
lack of convexity of the isentropes or slope discontinuities of the
entropy, which result in the appearance of composite waves, see \cite{MenikoffPlohr89}. 
To overcome the problem,
one could consider an approximate model by
means of a relaxation procedure. One obtains a Homogeneous Relaxation
model (HRM) where the relaxation towards the thermodynamical
equilibrium is driven by source terms which comply with the entropy growth
criterion. Section \ref{sec:homog-relax-model} presents two HRM models
depending on whether phase transition occurs or not, following the
construction proposed in \cite{BH05} (see also refer to \cite{Hurisse14,
  HelluyHurisse15, Hurisse17} for computational aspects). 


\section{A consistent thermodynamical description of the three phase
  system}
\label{sec:thermo}

The purpose of this section is to give a proper description of the
thermodynamical model.
We begin by the determination of the extensive constraints on the
state variables of the thermodynamical system. 
Because the gaseous phases are miscible with one another and
immiscible with the liquid, the volume constraint is
non-symmetric. According to the second principle, the mixture entropy
achieves its maximum at thermodynamical equilibrium. We characterize
two possible equilibria 
depending on phase transition
occurs or not between the liquid and its vapor. 
One recovers the Dalton's law satisfied by the gaseous phases.
Then  we introduce the intensive formulation and study the equilibrium
specific entropies for the models without and with phase transition.
It turns out that they are concave, possibly with a saturation zone.

\subsection{Single fluid thermodynamics: main definitions and assumptions}
\label{sec:single-fluid-therm}


Consider a fluid of mass $M\geq0$ and internal energy $E\geq 0$ occupying a
volume $V\geq 0$. As the fluid is homogeneous and at rest, its
thermodynamical behaviour is described by its entropy function
\begin{equation*}
  \begin{aligned}
    S : (\R^+)^3 &\to \R\\
    (M,V,E) &\mapsto S(M,V,E).
  \end{aligned}
\end{equation*}
This entropy function $S$ is concave with respect to $W=(M,V,E)\in (\R^+)^3$.
Then it is classical to extend it by
$-\infty$ outside the close convex cone $(\R^+)^3$
\begin{equation*}
  S(W) = 
  \begin{cases}
    S(W) , & W \in (\R^+)^3,\\
    -\infty, & \text{elsewhere}.
  \end{cases}
\end{equation*}
We adopt the assumptions stated in \cite{callen85} and \cite{Evans04}.
\begin{Hyp}
  \label{thm:hyp}
  Assume the entropy $S:(\R^+)^3 \to \R \cup \{-\infty \}$ is such that
  \begin{enumerate}[label=(\roman*)]
  \item\label{it:hyp1} the set of admissible states $C:=\{ W\in
    (\R^+)^3, \; S(W) >-\infty\}$ 
    is a non-empty close convex domain,
  \item\label{it:hyp2} $S$ is a concave function of $W$,
  \item \label{it:hyp3} $S$ is extensive or Positively Homogeneous of
    degree 1 (PH1), that is
    \begin{equation*}
      \forall \lambda \in \R_*^+, \forall W \in C, \quad S(\lambda W) =
      \lambda S(W),
    \end{equation*}
  \item \label{it:hyp4} $S$ is upper semi-continuous that is
    \begin{equation*}
      \forall W_0 \in C, \, \lim_{W\to W_0} \sup S(W) \leq S(W_0),
    \end{equation*}
  \item \label{it:hyp5} $S$ is of class $\mathcal C^2$ on $C$ and its partial
    derivative with respect to the internal energy is strictly
    positive
    \begin{equation*}
      \forall W \in C, \quad \dfrac{\p S}{\p E} >0.
    \end{equation*}
  \end{enumerate}
\end{Hyp}
Assumptions \ref{it:hyp2} and \ref{it:hyp3} are equivalent to assume
$(-S)$ sub-linear \cite{rockafellar}. The existence and continuity
assumption on the derivatives of $S$ is quite strong even if it is
common in literature. 
Observe that
the extensive entropy $S$ cannot be strictly concave since it is
PH1. 

The derivative of a PH1 function is PH0, said intensive.
Therefore the smoothness assumption \ref{it:hyp5} allows to define intensive potentials:
\begin{itemize}
\item the temperature $T$ 
  \begin{equation*}
    \dfrac{1}{T} = \dfrac{\p S}{\p E},
  \end{equation*}
\item the pressure $P$ 
  \begin{equation*}
    \dfrac{P}{T} = \dfrac{\p S}{\p V},
  \end{equation*}
\item the chemical potential $\mu$ 
  \begin{equation*}
    \mu= -T \dfrac{\p S}{\p M}.
  \end{equation*}
\end{itemize}
Hence one can state the extensive Gibbs relation
\begin{equation*}
  TdS= dE+PdV-\mu dM.
\end{equation*}
It is also common to define the specific entropy $s$ by
\begin{equation*}
M s = S(M,V,E).  
\end{equation*}
The extensive entropy $S$ being PH1, $s$ is PH0 (intensive) such that
\begin{equation}
  \label{eq:s_intensive}
  s=S(1,V/M,E/M).
\end{equation}
Hence $s$ can be seen as a function of the
specific volume $V/M=:\tau$ and the specific energy $E/M=:e$.
Setting $M=1$ in the extensive Gibbs relation gives the analogous
intensive form
\begin{equation}
\label{eq:Gibbs_intensive}
  Tds= de+Pd\tau.
\end{equation}
Since $S$ is PH1, it satisfies the Euler's relation $\nabla S\cdot
(M,V,E)^T=S$ which leads to another characterization of the chemical
potential
\begin{equation}
  \label{eq:mu}
  \mu=-Ts+p\tau+e.
\end{equation}
\subsection{Extensive description of the three-phase model}

We now consider a fluid system of fixed mass $M\geq 0$, volume $V\geq 0$ and
internal energy $E\geq 0$,
composed of a gas (indicated by the index $g$) and a pure body
present under its liquid phase (with index $l$) and its vapor phase
(with index $v$).
We assume that no mass transfer arises between the gas and the others
remaining phases but only mechanical and thermal exchanges.
We use the (abusive) appellation
\textit{phase} to indicate either the liquid, the vapor or the gaseous
component of the mixture.

We denote by $M_k\geq 0$, $V_k\geq 0$ and $E_k\geq 0$ the mass, the
volume and the internal energy of the phase
$k\in \{l,g,v\}$. We assume that each phase is entirely described by
its entropy function $S_k$ satisfying
Assumptions \ref{thm:hyp} for an extensive state vector
$W_k=(M_k,V_k,E_k)$ belonging to the close
convex cone $C_k$ defined in Assumption \ref{thm:hyp}-\ref{it:hyp1}.

We now state the constraints on the extensive variables.
By the mass conservation, one has
\begin{equation}
  \label{eq:mass}
  M=M_l+M_g+M_v,
\end{equation}
and the internal energy conservation leads to
\begin{equation}
  \label{eq:energy}
  E= E_l+E_v+E_g.
\end{equation}
The vapor is miscible with the gas, that is these two phases form an
intimate mixture occupying the same volume. On the other hand, the
liquid phase is immiscible with the gas and the
vapor, that is it occupies a different volume at a mesoscopic scale.
One gets the following volumic constraints
\begin{equation}
  \label{eq:volume}
  \begin{cases}
    V = V_l + V_v,\\
    V_g=V_v.
  \end{cases}
\end{equation}
Note that we assume that no vacuum can occur (otherwise, one should
consider $V\geq V_l+V_v$) and that the vapor and the gas are perfectly
intimate.
Unlike the mass or energy constraints, the volume constraint is not
invariant over permutation of the indexes $k=l,g,v$. This feature
induces difficulties to properly characterize the mixture equilibrium
and the mixture entropy
(both in extensive and intensive variables).

\begin{Rem}
  \label{rem:absence}
  If the vapor phase (resp. the gas) is absent, the
  system is made of the two remaining phases. To remove the vapor
  phase (resp. the gas), one has to impose
  $M_v=0$ (resp. $M_g=0$). Indeed setting $V_v=0$ (resp. $V_g=0$) is
  meaningless because the volume constraint \eqref{eq:volume} would
  impose the disappearance of both the vapor and the gas phases. On the
  other hand, if the liquid phase is absent, one has to 
  set both $M_l=0$ and $V_l=0$.
\end{Rem}

Let us address the definition of the extensive equilibrium entropy of
the mixture.
Out of equilibrium, the entropy of the three-phase system is
the sum of the phasic entropies. For
$(W_l,W_g,W_v)\in C_l \times C_g \times C_v$, it reads
\begin{equation}
  \label{eq:syst_entrop}
  \Sigma (W_l,W_g,W_v) = S_l(W_l)+S_g(W_g) + S_v(W_v).
\end{equation}
The second principle of Thermodynamics states that the system will
evolve until the entropy $\Sigma$ reaches a maximum.
Depending on whether or not mass transfer arises between the vapor and
the liquid phases, the maximization process relies on different set of
constraints, namely $\Omega_{ext}^{NPT}$ (No Phase Transition) and
$\Omega_{ext}^{PT}$ (Phase Transition), leading to two different mixture entropies.
We recall that phase transition is not allowed between the gas and the
other phases since  it has a different molecular structure.
Hence $M_g$ is fixed.

\begin{Def}
  Fix $M_g\geq 0$.
  Let $W=(M,V,E)\in (\R^+)^3$ be the state vector of the three-phase system.
  The equilibrium entropy of the mixture is:
  \begin{itemize}
  \item without phase transition: $M_l$ and $M_v$ are fixed
    satisfying the mass conservation  \eqref{eq:mass} and
    \begin{equation}
      \label{eq:Seq_NPT}
      S_{NPT}(M,V,E,M_l,M_g) =  \max_{(W_l,W_g,W_v)\in \Omega_{ext}^{NPT}} \Sigma
      (W_l,W_g,W_v) ,
    \end{equation}
    where $\Omega_{ext}^{NPT} := \{W_k \in C_k, \, k=l,g,v \ |
    \eqref{eq:energy} \text{ and } \eqref{eq:volume} \text{ hold}\}$  
  \item with phase transition: 
    \begin{equation}
      \label{eq:Seq_PT}
      S_{PT}(M,V,E,M_g) =  \max_{(W_l,W_g,W_v)\in \Omega^{PT}} \Sigma
      (W_l,W_g,W_v) ,
    \end{equation}
    where $\Omega_{ext}^{PT} := \{W_k\in C_k, \ k=l,g,v \ |
    M-M_g=M_l+M_v , 
    \eqref{eq:energy} \text{ and } \eqref{eq:volume} \text{ hold}\}$.
  \end{itemize}
\end{Def}

The constraint sets $\Omega_{ext}^{NPT}$ and $\Omega_{ext}^{PT}$ are  closed
bounded convex sets. 
According to Assumption \ref{thm:hyp}-\ref{it:hyp4} the entropies
$S_k$ are lower semi-continuous
functions. Then the optimization problem is 
well posed \cite{rockafellar, HUL01}.

\begin{Prop}
  \label{prop:ext_entrop}
  The extensive equilibrium entropy $S_{NPT}$ (resp. $S_{PT}$) of the three-phase mixture defined
  either by 
  \eqref{eq:Seq_NPT} (resp. \eqref{eq:Seq_PT}) is a PH1 concave function of
  its arguments.
\end{Prop}

\begin{proof}
  The function $\Sigma(W_l,W_g,W_v)$ is a concave function on
  $C_l\times C_g \times C_v$
  since it is a sum of concave functions.
  We now focus on the optimization problem \eqref{eq:Seq_NPT} over the
  set of constraints $\Omega_{ext}^{NPT}$ that is without
  phase transition. The mass of gas $M_g$ is fixed and the
  maximization is performed on the volume and the energy only.
  Hence we omit the dependency on $M_l$ and $M_g$ and get 
  \begin{equation*}
    \begin{aligned}
      S(M,V,E)
      &= \max_{
        \begin{cases}
          V=V_l+V_v\\
          V_g=V_v\\
          E=E_l+E_g+E_v
        \end{cases}
      }
      \Sigma      (W_l,W_g,W_v),\\
      &= \max_{
        \begin{cases}
          V=V_l+V_v\\
          E=E_l+E_g+E_v
        \end{cases}
      }
     \Sigma(W_l, (M_g,V_v,E_g), W_v).
    \end{aligned}
  \end{equation*}
  Since the masses $M$ and $M_k$, $k\in \{l,g,v\}$, are fixed, the
  problem can be written under the
  following form
  \begin{equation*}
    \begin{aligned}
      S(W)&=(\mathbf{A}H)(V,E)\\
      &= \max \{ H(V_l,V_v, E_l,E_g,E_v) |\;
      \mathbf{A} (V_l,V_v, E_l,E_g,E_v)^t = (V,E)^t\},
    \end{aligned}
  \end{equation*}
  where $H(V_l,V_g, E_l,E_g,E_v)=S_l(M_l,V_l,E_l) + S_g(M_g,V_v,E_g) +
  S_v(M_v,V_v,E_v)$ and
  $\mathbf{A}=
  \begin{pmatrix}
    1 & 1 & 0 & 0 & 0\\
    0 & 0 & 1 & 1 & 1
  \end{pmatrix}$
  is a linear mapping from $(\R^+)^5$ to $(\R^+)^2$ defining the
  constraints $V=V_l+V_v$ and
  $ E=E_l+E_g+E_v$.
  Because the function $H$ is concave with respect to $(V_l,V_v, E_l,E_g,E_v)\in (R^+)^5$
  (as the restriction of the concave function $\Sigma$) and $\mathbf{A}$ is a
  linear transformation, the function $\mathbf{A}H$ is also concave
  with respect to $(V,E)$ (see
  \cite{rockafellar}, Section 5). Then it follows that $S(W)$ is
  concave with respect to $W=(M,V,E)\in (\R^+)^3$. 
  Similar arguments hold in the case of phase transition between the
  liquid and its vapor.
\end{proof}

\begin{Rem}
  In \cite{HS06} the authors provide a similar extensive definition of the
  mixture entropy for a two-phase mixture when considering phase
  transition between the two phases, indexed by $k=1,2$. They consider
  the mass and energy conservation that is $M=M_1+M_2$ and
  $E=E_1+E_2$.
  As the volume constraint is considered, they distinguish the 
  immiscible and the miscible mixtures. When considering an immiscible
  mixture, their volume constraint is $V=V_1+V_2$. Then the extensive
  entropy of the mixture satisfies an analogous formulation as
  \eqref{eq:Seq_PT}, which turns to be an sup-convolution
  operation, namely
  \begin{equation*}
    S(W) = S_1\square S_2(W) = \max_{W_1\in C_1} (S_1(W_1)+S_2(W-W_1)),
  \end{equation*}
  where the symbol $\square$ is a notation for sup-convolution in
  convex analysis.
  When considering a miscible approach, their volume constraint is
  $V=V_1=V_2$.
  Here again the extensive entropy of the mixture is a
  sup-convolution operation.
  The sup-convolution operation turns to have many interesting
  properties (especially linked to the Legendre transform). Such
  properties have been studied in \cite{HM10} and \cite{Mathis10}, for the
  computation of admissible pressure laws for immiscible and
  miscible binary mixture.

  In the present case, because the volume constraint \eqref{eq:volume}
  is simultaneously immiscible (between the liquid and the vapor and gas
  phases) and miscible (between the gas and the vapor), we cannot
  express the energy of the mixture as a sup-convolution procedure.
\end{Rem}

When the equilibrium entropy without phase transition is
differentiable with respect to the volume $V$ and the internal energy
$E$, then one can define the temperature and the pressure of the
mixture at equilibrium
\begin{equation}
  \label{eq:TP_eq_NPT}
  \dfrac{1}{T} = \dfrac{\p S_{NPT}}{\p E}(M,V,E, M_l, M_g), \quad
  \dfrac{P}{T} = \dfrac{\p S_{NPT}}{\p V}(M,V,E, M_l, M_g).
\end{equation}
The chemical potential and the potentials linked to the masses $M_l$
and $M_g$ are
\begin{equation*}
  \dfrac{\mu}{T} =- \dfrac{\p S_{NPT}}{\p M}(M,V,E, M_l, M_g), \quad
  \dfrac{\lambda_k}{T} = \dfrac{\p S_{NPT}}{\p M_k}(M,V,E, M_l, M_g),\quad k=l,g.
\end{equation*}
Hence one has the following relation
\begin{equation*}
  TdS_{NPT} = dE + pdV -\mu dM+\lambda_l dM_l + \lambda_g dM_g.
\end{equation*}
When phase transition is considered, one gets
\begin{equation*}
  TdS_{PT} = dE + pdV -\mu dM+ \lambda_g dM_g.
\end{equation*}

When the maximum of the mixture entropy is reached 
in the interior of the set of
constraints, the three phases are present and at thermodynamical
equilibrium \cite{HS06, Mathis10}.
\begin{Prop}
\label{prop:eq_thermo}
  The thermodynamical equilibrium corresponds to
  \begin{itemize}
  \item the equality of the temperatures
    \begin{equation}
      \label{eq:temp_eq}
      T_l=T_g=T_v,
    \end{equation}
  \item the Dalton's law on the pressures of the gas and the vapor phases
    \begin{equation}
      \label{eq:dalton}
      p_l = p_g+p_v.
    \end{equation}
  \end{itemize}
  Moreover if phase transition is allowed between the liquid and its vapor then
  the equilibrium is also characterized by
  \begin{equation}
    \label{eq:mu_eq}
    \mu_l=\mu_v.
  \end{equation}
\end{Prop}

\begin{proof}
The optimization with respect to the energy and the volume are the
same on the two sets of constraints $\Omega_{ext}^{NPT}$ and $\Omega_{ext}^{PT}$.
Let us fix the energy $E_k$ of the phase $k\in \{l,g,v\}$. Then
$E-E_k= E_{k'}+ E_{k''}$, with $ k'\neq k'', \; k, k'\in
\{l,g,v\}$. Thus
\begin{equation*}
  \dfrac{\p}{\p E_{k'}}(S_k (M_k,V_k,E_k)+ S_{k'}
  (M_{k'},V_{k'},E_{k'})+ S_{k''}(M_{k''},V_{k''},E_{k''})) 
  = \dfrac{1}{T_{k'}}-\dfrac{1}{T_{k''}}.
\end{equation*}
Then the maximum is reached for $T_{k'}=T_{k''}$ for any $k'\neq k''
\in \{l,g,v\}$.
Optimizing with respect to the volume under the volume constraint
\eqref{eq:volume} gives
\begin{equation*}
    \dfrac{\p}{\p V_l}(S_l (M_l,V_l,E_l)+ S_{g}
  (M_{g},V-V_{l},E_{g})+ S_{v}(M_{v},V-V_{l},E_{g})) 
  = \dfrac{p_l}{T_{l}}-\left( \dfrac{p_g}{T_{g}}+  \dfrac{p_v}{T_{v}}\right).
\end{equation*}
Since the temperature are equal, it yields the Dalton's law on the pressures.
We now focus on the case where phase transition occurs.
In the case of phase transition, we then optimize with respect to the mass in the set
of constraints $\Omega_{ext}^{PT}$. Since the mass of the gas $M_g$ is fixed, one
has $M-M_g = M_l+M_v$.
It yields
\begin{equation*}
    \dfrac{\p}{\p M_l}(S_l (M_l,V_l,E_l)+ S_{g}
  (M_{g},V_g,E_{g})+ S_{v}(M-M_g-M_l,V_{v},E_{v})) 
  = \dfrac{\mu_l}{T_{l}}- \dfrac{\mu_v}{T_{v}}.
\end{equation*}
Because $T_l=T_v$, the chemical potentials of the liquid and vapor
phases are also equal, the chemical potential of the gas $\mu_g$ being fixed.
\end{proof}
One observe that the pressure relation \eqref{eq:dalton}, which
contains the Dalton's law on the miscible vapor and gaseous phases, is
a direct consequence of the maximization process under the volumic
constraint \eqref{eq:volume}.

As a consequence, at equilibrium, one may define the
mixture temperature $T$ and pressure $p$ by
\begin{equation}
  \label{eq:p_T_eq_mixture}
  \begin{aligned}
    T&=T_l=T_g=T_v\\
    p &= p_l = p_g+p_v.
  \end{aligned}
\end{equation}
Nevertheless it is not possible de define a mixture chemical potential.

\subsection{Intensive characterization of the entropies}
\label{sec:sec_int}

We now turn to the definition of intensive quantities.
The system is now entirely described by its intensive entropy
$s$ defined by \eqref{eq:s_intensive} as a function of the specific
volume $\tau=V/M>0$ and the 
specific internal energy $e=E/M>0$. 

We introduce the mass fraction $\varphi_k$, the volume fraction
$\alpha_k$ and the energy fraction $z_k$ of the phase $k\in \{l,g,v\}$
defined respectively by
\begin{equation}
  \label{eq:fractions}
  \varphi_k=M_k/M, \quad \alpha_k=V_k/V, \quad z_k=E_k/E,
\end{equation}
which belong to $[0,1]$.
Each phase $k=l,g,v$ has a specific volume 
$\tau_k=V_k/M_k=\alpha_k\tau/\varphi_k$
and a specific internal energy $e_k=E_k/M_k=z_ke/\varepsilon_k$.
The specific entropy $s_k$ of the phase $k\in \{l,g,v\}$ is defined by
\begin{equation*}
  s_k(\tau_k,e_k) = S_k(1,\tau_k,e_k).
\end{equation*}
Moreover one can derive the intensive form of the Gibbs relation
\eqref{eq:Gibbs_intensive} for each phase
$k\in \{l,g,v\}$ 
\begin{equation}
  \label{eq:gibbs_int_k}
  T_k ds_k = de_k+p_kd\tau_k.
\end{equation}

We now turn to the intensive formulation of the extensive constraints.
The extensive volume constraint \eqref{eq:volume} translates into
\begin{equation}
  \label{eq:alfk}
  \begin{cases}
    1= \alpha_l + \alpha_v,\\
    \alpha_g=\alpha_v.
  \end{cases}
\end{equation}
The mass and energy conservations \eqref{eq:mass} and
\eqref{eq:energy} read 
\begin{eqnarray}
  \label{eq:phik}
      1&=&\varphi_l+\varphi_g+\varphi_v,\\
  \label{eq:zk}
      1&=&z_l+z_g+z_v.     
\end{eqnarray}
Out of equilibrium the intensive entropy of the three-phase system,
expressed as a function of $\tau$, $e$ and the fractions
$\varphi_k,\alpha_k, z_k$, $k=l,g,v$,
reads
\begin{equation}
  \label{eq:intensive_entropy}
  \begin{aligned}
    &\sigma(\tau,e,(\varphi_k)_k ,(\alpha_k)_k, (z_k)_k) \\
    &= \varphi_L s_l \left( \dfrac{\alpha_l}{\varphi_l}\tau,
      \dfrac{z_l}{\varphi_l}e\right) + \varphi_g s_g \left(
      \dfrac{\alpha_g}{\varphi_g}\tau, \dfrac{z_g}{\varphi_g}e\right)
    + \varphi_v s_v \left( \dfrac{\alpha_v}{\varphi_v}\tau,
      \dfrac{z_v}{\varphi_v}e\right).
  \end{aligned}
\end{equation}
At equilibrium and at fixed $(\tau,e)$, the intensive entropy reaches
its maximum. As in the extensive formulation, one has to define the
set of constraints depending on whether phase transition occurs or not between
the liquid and the vapor phase. Moreover since the gaseous phase does
not exchange mass with the others phases, its mass fraction
$\varphi_g$ is fixed during the optimization process.
 
\begin{Prop}
\label{prop:intensive_s}
  Fix $\varphi_g\in [0,1]$.
  Let $(\tau,e)\in (\R^+)^2$ be the specific state vector of the
  system.
  The equilibrium intensive entropy $s$ of the mixture is:
  \begin{itemize}
  \item without phase transition: $\varphi_l$, $\varphi_g$ are fixed
    according to \eqref{eq:phik} and
    \begin{equation} 
      \label{eq:intensive_s_NPT}
      s_{NPT}(\tau,e, \varphi_l, \varphi_g) = \max_{
        ((\alpha_k)_k,(z_k)_k) \in \Omega_{int}^{NPT}
      }
      \sigma(\tau,e,(\varphi_k)_k ,(\alpha_k)_k, (z_k)_k),
    \end{equation}
    where $\Omega_{int}^{NPT}:=\{ (\alpha_k, z_k), \ k=l,g,v |\
    \eqref{eq:alfk} \text{ and } \eqref{eq:zk} \text{ hold}\}$
  \item with phase transition: 
    \begin{equation}
      \label{eq:intensive_s_PT}
      s_{PT}(\tau,e,\varphi_g) = \max_{
        ((\varphi_k)_k,(\alpha_k)_k,(z_k)_k) \in \Omega_{int}^{PT}
      }
      \sigma(\tau,e,(\varphi_k)_k ,(\alpha_k)_k, (z_k)_k),
    \end{equation}
    where $\Omega_{int}^{PT}:= \{ (\varphi_k,\alpha_k,z_k), \ k=l,g,v
    | 1-\varphi_g = \varphi_l+\varphi_v, \, \eqref{eq:alfk} \text{ and
    } \eqref{eq:zk} \text{ hold}\}$.
  \end{itemize}
  In both cases the equilibrium intensive entropy is a concave
  function of its arguments.
\end{Prop}

\begin{proof}
  The characterization of the mixture intensive entropy is a direct
  consequence of the homogeneity of the extensive mixture entropies
  defined in \eqref{eq:Seq_NPT} and \eqref{eq:Seq_PT}, see
  Proposition \ref{prop:ext_entrop}.
  The relation \eqref{eq:intensive_s_NPT} (resp.
  \eqref{eq:intensive_s_PT})
  is achieved by dividing the optimization problem
  \eqref{eq:Seq_NPT} on the set of constraints $\Omega_{ext}^{NPT}$
  (resp. \eqref{eq:Seq_PT} on the set of constraints $\Omega_{ext}^{PT}$)
  by the mass $M$.
  In the case without phase transition, 
 the intensive entropy $s_{NPT}(\tau,e,\varphi_l,\varphi_g)$ is the
 restriction of the extensive entropy $S_{NPT}(M,V,E, M_l, M_g)$ on the
 affine convex
  subset $\{1\}\times (\R_+)^2\times [0,1]^2$.
  Since $S_{NPT}$ is a concave function of $(M,V,E)$, $s_{NPT}$ is a concave
  function of $(\tau,e)$. The same holds in the case of phase transition.
\end{proof}


We now focus on the intensive equilibrium entropy without phase
transition $s_{NPT}$ and prove that it is strictly convex with respect to
$(\tau,\varepsilon)$.


\begin{Prop}
  \label{prop:entrop_convex_NPT}
  Assume that the mass fractions $\varphi_k$ are fixed (no phase
  transition is allowed). Then the intensive equilibrium entropy
  \eqref{eq:intensive_s_NPT}
  \begin{itemize}
  \item depends only on $(\tau,e)$ 
  \item is a strictly concave function of $(\tau,e)$
  \item satisfies the relation : $Tds_{NPT} = de + pd\tau$, where $T$ and
    $p$ are the mixture temperature and pressure at equilibrium.
  \end{itemize}
\end{Prop}

\begin{proof}
According to the definition \eqref{eq:intensive_s_NPT}, it is obvious that
the equilibrium mixture entropy
depends only on $(\tau,e)$ at fixed mass
fractions $\varphi_k$, $k=l,g,v$. 
Then for any equilibrium state $(\tau,e)\in (\R^+)^2$, it exists $(\tau_k,e_k)$
such that 
\begin{equation}
  \label{eq:tau_e}
  \begin{cases}
    e = \varphi_l e_l +\varphi_g e_g+\varphi_v e_v\\
    \tau = \varphi_l \tau_l + \varphi_g\tau_g \\
    \varphi_v \tau_v  = \varphi_g \tau_g.
  \end{cases}
\end{equation}

We now prove the Gibbs relation 
\begin{equation*}
  Tds =de + pd\tau,
\end{equation*}
at fixed $\varphi_k$.
The phasic entropies satisfy
\begin{equation*}
  Tds_k =de_k + p_kd\tau_k, \quad k=l,g,v.
\end{equation*}
Multiplying by $\varphi_k$ and summing over $k=l,g,v,$ give
\begin{equation*}
  T d (\sum_{k=l,g,v} \varphi_k s_k)= d (\sum_{k=l,g,v} \varphi_k e_k) +
  \sum_{k=l,g,v}\varphi_k p_k d\tau_k,
\end{equation*}
according to the equality of the temperatures.
By \eqref{eq:intensive_entropy} and \eqref{eq:tau_e} it yields
\begin{equation*}
  Tds_{NPT} =de + p_l d(\varphi_l\tau_l) + p_gd(\varphi_g\tau_g)+p_v d(\varphi_v\tau_v).
\end{equation*}
We now use miscibility of the vapor and gas phases
$\varphi_g\tau_g=\varphi_v\tau_v$ to get
\begin{equation*}
  Tds_{NPT} =de + p_l d(\varphi_l\tau_l) + (p_g+p_v)d(\varphi_g\tau_g).
\end{equation*}
The characterization of the pressure equilibrium leads to
conclusion. 

We turn to the strict concavity of the $s_{NPT}$. Since the
$\varphi_k$, $k=l,g,v,$, are fixed, we denote 
$s_{NPT}(\tau,e) = s_{NPT}(\tau,e, (\varphi_k)_k)$.
In order to prove that the entropy is strictly concave, we show that
for any equilibrium states $(\tau,e)$ and $(\tau',e')$ in $(\R^+)^2$, one has
\begin{equation*}
  s_{NPT}(\tau,e) < s_{NPT}(\tau',e')+ \nabla_{(\tau,e)} s_{NPT}(\tau',e')\cdot
  \begin{pmatrix}
    \tau-\tau'\\ e-e'
  \end{pmatrix}.
\end{equation*}
Using formulation
\eqref{eq:intensive_entropy} one has
\begin{equation*}
  s_{NPT}(\tau,e) = \varphi_l s_l(\tau_l,e_l) + \varphi_g s_g
  (\tau_g,e_g)+\varphi_v s_v(\tau_v,e_v).
\end{equation*}
  Since the phasic entropies are strictly concave functions of
  $(\tau_k,e_k)$ and differentiable,  there exists $(\tau_k',e_k')$
  such that
\begin{equation*}
  \begin{aligned}
    s_{NPT}(\tau,e) < \sum_{k=l,g,v}\varphi_k s_k (\tau_k', e_k') + \varphi_k\nabla s_k
    (\tau_k',e_k')\cdot
    \begin{pmatrix}
      \tau_k-\tau_k'\\e_k-e_k'
    \end{pmatrix}.
  \end{aligned}
\end{equation*}
Now one has $\nabla s_k (\tau_k',e_k') =
\begin{pmatrix}
  1/T_k (\tau_k',e_k') \\ p_k (\tau_k',e_k') /T_k (\tau_k',e_k')
\end{pmatrix}$
with $T_k(\tau_k',e_k')=T$, $\forall k=l,g,v$, see Proposition
\ref{prop:eq_thermo}.
The definition of the equilibrium entropy, the equality of the
temperature, and the constraints \eqref{eq:phik} lead to
\begin{equation*}
  \begin{aligned}
    s_{NPT}(\tau,e)&< s_{NPT} (\tau',e') + \dfrac{1}{T}(e-e')+\\
    &\dfrac{1}{T} \left( \varphi_l p_l (\tau_l-\tau_l') + \varphi_g p_g
      (\tau_g-\tau_g') + \varphi_vp_v (\tau_v-\tau_v') \right) .
  \end{aligned}
\end{equation*}
Using the Dalton's law \eqref{eq:p_T_eq_mixture},
 one can express the mixture pressure as
$p_l=p_v+p_g=p$. 
Then the volume constraints $\tau = \varphi_l\tau_l+\varphi_g\tau_g$
and $\tau = \varphi_l\tau_l+\varphi_v\tau_v$ give
\begin{equation*}
  s_{NPT}(\tau,e)< s_{NPT} (\tau',e') + 
    \dfrac{1}{T_l}(e-e')+ \dfrac{p}{T}
    (\tau-\tau') .
\end{equation*}
According to the Gibbs relation, one has $\nabla_{(\tau,e)} s_{NPT} =
\begin{pmatrix}
  1/T,\\
  p/T
\end{pmatrix}$ which leads to the conclusion.
\end{proof}

As phase transition is considered between the liquid and its vapor,
the mixture entropy is no longer strictly concave with respect to
$(\tau,e)$ as $\varphi_g$ is fixed. 
\begin{Prop}
  \label{prop:entropy_PT}
  Assume that the mass fraction $\varphi_g$ is fixed. Then the
  intensive equilibrium entropy \eqref{eq:intensive_s_PT} 
  \begin{itemize}
  \item depends only on $(\tau,e)$
  \item satisfies the relation : $Tds_{PT}=de+pd\tau$, where $T$ and $p$
    are the mixture temperature and pressure at equilibrium.
  \end{itemize}

\end{Prop}
The proof is similar to the proof of the Proposition
\ref{prop:entrop_convex_NPT}. 
\begin{Rem}
  \label{rem:PT}
  Note that we do not prove that the equilibrium entropy is a strictly
  concave function of $(\tau,e)$. Actually this is not the case for
  binary (immiscible) mixture, see \cite{Jaouen01, HS06, HM10} for
  instance.  One may find the computation of a three-phase mixture
  pressure law in \cite{BMHM12} (with a mix type volume constraint
  like \eqref{eq:volume}) . The authors consider that each phase is
  depicted by a stiffened gas but it is not possible to give an
  analytical formulation of the pressure. However computational
  results illustrate that a saturation zone exists, that is the
  mixture entropy is not strictly concave.
\end{Rem}


\section{Equilibrium three-component Euler systems}
\label{sec:modell-assumpt}

We now take into account the dynamic of the three-phase mixture,
assuming that the three phases have the same velocity.
The aim of this section is to provide an homogeneous equilibrium multicomponent
Euler's system, called HEM model, with appropriate closure laws in agreement with the
thermodynamical equilibria studied in Section \ref{sec:sec_int}. 
Two HEM models are presented corresponding to the cases with or without
phase transition.
The models have good properties: entropy structure and hyperbolicity.

\subsection{Three-phase model without phase transition}
\label{sec:euler_NPT}

At thermodynamical equilibrium the three phase flow is depicted by the
multicomponent Euler system
\begin{equation}
  \label{eq:euler_NPT}
  \begin{cases}
    \p_t (\varphi_l \rho) + \p_x (\varphi_l \rho
    u) = 0,\\
    \p_t (\varphi_g \rho) + \p_x (\varphi_g \rho
    u) = 0,\\
    \p_t \rho + \p_x(\rho u) =0, \\
    \p_t (\rho u) + \p_x (\rho u^2 + p) =0,\\
    \p_t (\rho E ) + \p_x ((\rho E + p)u) =0,\\
    E= \dfrac 1 2 u^2 +e,\\
    \forall k\in \{l,g,v\} : \; p_k=p_k(\tau_k,e_k), \; \tau_k =
    \rho_k^{-1},\\
    \varphi_l+\varphi_g+\varphi_v=1,
  \end{cases}
\end{equation}
where the flow as a density $\rho$ (we also define the specific volume
$\tau=1/\rho$), a velocity $u$, a pressure $p$,
and an internal energy $e$, $E$ being the total energy.
The phase $k=l,g,v$ is depicted by its mass fraction, its pressure
$p_k$, its specific volume $\tau_k$ and its specific internal energy $e_k$, 
see Section \ref{sec:sec_int}. All the phases evolve at the same
velocity $u$ and we recall that 
\begin{equation}
\label{eq:intensive_cons}
  \begin{cases}
    e = \varphi_l e_l+\varphi_v e_v +\varphi_g eg,\\
    \tau = \varphi_l \tau _l + \varphi_g \tau_g,\\
    \varphi_v\tau_v = \varphi_g \tau_g.
  \end{cases}
\end{equation}
The multicomponent Euler system admits ten equations and has seventeen unknowns
which are
$$(\rho, u,E,p,e, (\varphi_k)_{k\in \{l,g,v\}}, (\tau_k) _{k\in
  \{l,g,v\}}, (e_k) _{k\in \{l,g,v\}},(p_k) _{k\in \{l,g,v\}}).$$
Thus one has to provide seven closure laws.
The first three closure laws are given by the constraints \eqref{eq:intensive_cons}.

The 4 remaining closure laws are given by Proposition
\ref{prop:eq_thermo}, that is
\begin{equation}
  \label{eq:euler_eq_NPT}
  \begin{cases}
    T=T_l=T_g=T_v,\\
    p= p_l=p_g+p_v,
  \end{cases}
\end{equation}
where $T$ and $p$ are the thermodynamical temperature and pressure of
the three phase flow and $p=p(1/\rho,e,\varphi_l,\varphi_g)$.

\begin{Prop}
  \label{prop:transport_s}
  The intensive entropy $s_{NPT}(\tau,e,\varphi_l,\varphi_g)$ defined by
  \eqref{eq:intensive_s_NPT} satisfies 
  \begin{equation}
    \label{eq:transport_s}
    \p_t s + u\p_xs_{NPT} =0.
  \end{equation}
\end{Prop}
\begin{proof}
  Let $U=(\rho,\rho u, \rho E,\varphi_g\rho, \varphi_l \rho)$ is a
  smooth solution of the  system \eqref{eq:euler_NPT}, then one has
  \begin{equation*}
    \begin{aligned}
      \p_t \tau + u\p_x \tau -\tau \p_x u &=0,\\
      \p_t u + u\p_x u + \tau\p_x p&=0,\\
      \p_t e+ u \p_x e+ p\tau\p_x u&=0,\\
      \p_t \varphi_k + u \p_x\varphi_k&=0, \quad k=l,g.
    \end{aligned}
  \end{equation*}
  Since $s_{NPT}$ is function of $(\tau,e,\varphi_l,\varphi_g)$, it follows
  \begin{equation*}
    \begin{aligned}
      \p_t s_{NPT} &= \dfrac{\p s_{NPT}}{\p \tau} \p_t \tau + \dfrac{\p s_{NPT}}{\p e}
      \p_t e + \dfrac{\p s_{NPT}}{\p \varphi_l\tau} \p_t \varphi_l +
      \dfrac{\p s_{NPT}}{\p \varphi_g} \p_t \varphi_g\\
      & = \p_x u \left( \tau\dfrac{\p s_{NPT}}{\p \tau}- \tau p \dfrac{\p
          s_{NPT}}{\p e}\right) - u \p_x s_{NPT}.
    \end{aligned}
  \end{equation*}
 Because the entropy $s_{NPT}$ satisfies the relation $Tds_{NPT}=de+pd\tau$
  (see Proposition \ref{prop:entrop_convex_NPT}), the
  first term of the right hand side is zero. Hence the
  entropy $s_{NPT}$ satisfies a transport equation.
\end{proof}

In order to study the hyperbolicity of the model \eqref{eq:euler_NPT},
we adapt a result given in \cite{Lagoutiere00} which extends the Godunov-Mock
theorem. 
\begin{Lem}
\label{lem:lem1}
  Let $w:\R^+ \times \R \to \R^n$ and $f:\R^n\to \R^n$ defining the
  system of conservation laws
  \begin{equation*}
    \p_t w(t,x) + \p_x f(w)(t,x)=0,
  \end{equation*}
  where $w=(w_1,w_2)^t$ with $w_1\in \R^l$ and $w_2\in\R^{n-l}$
  and $f=(0,f_2)^t$ with $f_2\in \R^{n-l}$. Assume that  $\eta(w)$ is
  a strictly convex function with respect to $w_2$ at fixed $w_1$ such
  that
  \begin{equation*}
    \p_t \eta(w) = 0,
  \end{equation*}
  and that $\nabla_{w_1} f_2(w)=0$.
  Then the system is hyperbolic.
\end{Lem}

\begin{proof}
  To prove the hyperbolicity we show that the system is symmetrizable
  that is there exists a symmetric positive-definite matrix $P$
  and a symmetric matrix $Q$
  such that
  \begin{equation*}
    P(w) \p_t w + Q(w) \p_x w=0.
  \end{equation*}
  We define the $n\times n$ symmetrization matrix $P(w)$ by
  \begin{equation*}
    P(w) =
    \begin{pmatrix}
      \mathbf{I}_{l} & 0\\
      0 & \nabla_{w_2}^2 \eta
    \end{pmatrix}.
  \end{equation*}
The entropy $\eta$ being strictly convex with respect to $w_2$, the
matrix $P(w)$ is symmetric positive-definite.
The associated convection matrix is $Q(w)= P(w) \nabla_w f(w)$.
Since $\nabla_{w_1} f_2(w)=0$, the matrix $Q$ is symmetric so that the
system is symmetrizable. As a consequence the system is hyperbolic.
\end{proof}

This lemma holds for any variables $(t,x)$ as soon as the system is
conservative. Besides we use it in Lagrangian coordinates to
prove the following result.

\begin{Thm}
  The system \eqref{eq:euler_NPT} is hyperbolic.
\end{Thm}

\begin{proof}
  First the system \eqref{eq:euler_NPT} can be written in Lagrangian
  coordinates
  \begin{equation*}
    \begin{cases}
      D_t \varphi_l = 0,\\
      D_t \varphi_g = 0,\\
      D_t \tau -D_m u =0,\\
      D_t u + D_m  p =0,\\
      D_t E + D_m (pu)=0,
    \end{cases}
  \end{equation*}
  where $D_t v=\p_t v + u\p_x v$ and $D_m v= \tau\p_x v$.
  The associated flux reads $f=(0,0,-u,p,pu)$.
  We introduce the function $\eta$
  \begin{equation*}
    \eta:(\varphi_l,\varphi_g,\tau,u,E)\to
    -s_{NPT}(\tau,E-u^2/2,\varphi_l,\varphi_g).
  \end{equation*}
  According to Proposition
  \ref{prop:entrop_convex_NPT}, 
  the function $s_{NPT}$ is strictly concave with respect to $(\tau,e)$ and
  depends only on $(\tau,e)$.
  Then $\eta$ is strictly convex
  with respect to $(\tau,u,E)$, see \cite{croisille91,GodRav91}.
  Moreover $s_{NPT}$ is solely advected by the system, since it satisfies
  \eqref{eq:transport_s}, see Proposition \ref{prop:transport_s}.
  Hence it yields
  \begin{equation*}
    \begin{aligned}
      D_t \eta(w) &= \p_\tau D_t \tau+\p_u D_t u + \p_E D_t E\\
      &= -\p_\tau s_{NPT}(\tau,
      E-u^2/2, \varphi_l,\varphi_g) D_t \tau + \\
      &\quad (u D_t u + D_t E)\p_e s_{NPT}(\tau,
      E-u^2/2, \varphi_l,\varphi_g)\\
      &= - \dfrac{p}{T} D_t \tau + \p_e s_{NPT} (u D_t u -D_t E)
      = - \dfrac{p}{T} D_m u - \dfrac u T D_m p + \dfrac 1 T D_m (p
      )=0.
    \end{aligned}
  \end{equation*}
  In addition the mixture pressure $p$, being a partial derivative of
  the entropy mixture $s_{NPT}$, does not depend on the fractions
  $\varphi_l$ and $\varphi_g$. It implies that
  $\nabla_{\varphi_l,\varphi_g} f=0$.
  Now Lemma \ref{lem:lem1} leads to the conclusion.
\end{proof}

\subsection{Three-phase model with phase transition}
\label{sec:euler_PT}
When phase transition occurs, the equilibrium multicomponent Euler
system reads
\begin{equation}
  \label{eq:euler_PT}
  \begin{cases}
    \p_t (\varphi_g \rho) + \p_x (\varphi_g \rho
    u) = 0,\\
    \p_t \rho + \p_x(\rho u) =0, \\
    \p_t (\rho u) + \p_x (\rho u^2 + p) =0,\\
    \p_t (\rho E ) + \p_x ((\rho E + p)u) =0,\\
    E= \dfrac 1 2 u^2 +e,\\
    \forall k\in \{l,g,v\} : \; p_k=p_k(\tau_k,e_k), \; \tau_k =
    \rho_k^{-1},\\
    \varphi_l+\varphi_g+\varphi_v=1.
  \end{cases}
\end{equation}
The system admits nine equations and seventeen unknowns which are
$$(\rho, u,E,p,e, (\varphi_k)_{k\in \{l,g,v\}}, (\tau_k) _{k\in
  \{l,g,v\}}, (e_k) _{k\in \{l,g,v\}},(p_k) _{k\in \{l,g,v\}}).$$
Thus one has to provide eight closure laws.

As in the previous case, three closure laws are given by the three intensive
constraints \eqref{eq:intensive_cons}
\begin{equation*}
  \begin{cases}
    e &= \varphi_l e_l+\varphi_v e_v +\varphi_g eg,\\
    \tau &= \varphi_l \tau _l + \varphi_g \tau_g,\\
    \varphi_v\tau_v &= \varphi_g \tau_g.
  \end{cases}
\end{equation*}
The five remaining closures are given by
Proposition \ref{prop:eq_thermo}
\begin{equation}
  \label{eq:euler_eq_PT}
  \begin{cases}
    T=T_l=T_g=T_v,\\
    p= p_l=p_g+p_v,\\
    \mu_l=\mu_v.
  \end{cases}
\end{equation}
where $T$ and $p$ are the thermodynamical temperature and pressure
of the three phase flow and $p=p(1/\rho,e,\varphi_g)$.

Since the equilibrium entropy is not a strictly concave function of its
arguments (see Remark \ref{rem:PT}), 
it is not possible to invocate the Godunov-Mock theorem or
its extension Lemma \ref{lem:lem1} to prove the hyperbolicity of the system.
However it is possible to prove the hyperbolicity by studying
the eigenvalues of the system
and the positivity of the mixture temperature.
\begin{Thm}
  The system \eqref{eq:euler_PT} is hyperbolic.
\end{Thm}
\begin{proof}
  The quasilinear form of the system \eqref{eq:euler_PT} reads
  \begin{equation*}
    \p_t
    \begin{pmatrix}
      \varphi_g\\\ \rho \\ u \\ e
    \end{pmatrix}
    +
    \begin{pmatrix}
      u & 0 & 0 & 0 \\
      0 & u & \rho & 0\\
      \dfrac{1}{\rho} \dfrac{\p p}{\p \varphi_g} &  \dfrac{1}{\rho}
      \dfrac{\p p}{\p \rho} & u &  \dfrac{1}{\rho} \dfrac{\p p}{\p
        \rho} \\
      0 & 0 & p/\rho & u
    \end{pmatrix}
    \p_x 
   \begin{pmatrix}
      \varphi_g\\ \rho \\ u \\ e
    \end{pmatrix}
    =0.
  \end{equation*}
The Jacobian matrix of the flux has four eigenvalues $u-c$, $u$
(double), $u+c$, where $c$ is the speed of sound given by
\begin{equation}
  \label{eq:c}
  c^2 /\tau^2= p \p_e p -\p_\tau p =-T (p^2 (s_{PT})_{ee} - 2 p (s_{PT})_{\tau e}+
  (s_{PT})_{\tau \tau}).
\end{equation}
According to Proposition \ref{prop:entropy_PT}, the entropy $(s_{PT})$ is a
concave function which depends only on $(\tau,e)$ at fixed
$\varphi_g$. Hence the right-hand side of \eqref{eq:c} is non negative
as soon as the temperature $T>0$. This concludes the proof.
\end{proof}


\section{Homogeneous Relaxation Models for the three-phase flow}
\label{sec:homog-relax-model}

The equilibrium multicomponent Euler systems, presented in the previous
section, are difficult to use for practical computations. Although
they are proved to be hyperbolic, their pressure laws have no
analytical expressions (even if $p_k$, $k=l,g,v$ are perfect gas
laws). Moreover it is well known, see for instance
\cite{MenikoffPlohr89, Jaouen01,BH05}, that such pressure laws present
pathologies such that slope discontinuities, lack of convexity of the
isentropes, leading to composite waves.
To overcome this problem, some authors proposed \cite{BH05,HS06,
  Hurisse14, HelluyHurisse15, Hurisse17} to approximate the equilibrium
Euler system by a homogeneous relaxation model.
It consists in adding convection equations on the fractions and to
modify the pressure to make it depend on the fractions. In order to
achieve the thermodynamical equilibrium, appropriate relaxation source terms
complete the equations on the fractions. The numerical
approximation of the relaxed model is easier. Traditionally it
consists on a splitting approach. In a first step the convective part
is treated with a approximate Riemann solver. During the second step
the conservative variables are stored and the pressure is updated from
the physical entropy maximization. By construction both steps are
entropy satisfying.

We propose in this section to construct the HRM models associated to
the HEM three-phase models studied in Section
\ref{sec:modell-assumpt} while distinguishing the cases where phase
transition occurs or not.
First we focus on the model without phase transition and adapt the
construction of the HRM model introduced in \cite{BH05}. The case with
phase transition is treated as corollary.

\subsection{HRM model without phase transition}
  \label{sec:hrm-PT}

Starting from the equilibrium three-phase model \eqref{eq:euler_NPT},
we propose a non-homogeneous model in which the three phases are no
longer at thermal and mechanical equilibrium (still without phase
transition). To do so one introduces supplementary variables that are
the volume fraction of liquid $\alpha_l$ and the energy fractions
$z_l$ and $z_g$ defined in \eqref{eq:fractions}.
Hence the pressure depends not only on $\rho,e, \varphi_l,\varphi_g$
but also on $Y=(\alpha_l, z_l,z_g)$.
When no mass transfer occurs between the liquid and the gas, the
fractions should be perfectly convected \textit{i.e.}
\begin{equation}
  \label{eq:fraction_conv}
  \p_t Y + u\p_x Y = 0.
\end{equation}
The mass conservation allows to write \eqref{eq:fraction_conv} under
the conservative form
\begin{equation}
  \label{eq:fraction_conv_cons}
  \p_t (\rho Y) + \p_x (\rho u Y) = 0.
\end{equation}
Thus the resulting HRM model reads
\begin{equation}
  \label{eq:HRM_NPT}
  \begin{cases}
    \p_t (\varphi_k \rho) +   \p_x (\varphi_k \rho u) =0, \quad
    k=l,g,\\
    \p_t (z_k \rho) +   \p_x (z_k \rho u) =0, \quad
    k=l,g,\\
    \p_t (\alpha_l\rho) +   \p_x (\alpha_l\rho u) =0, \\
    \p_t \rho + \p_x (\rho u) =0,\\
    \p_t (\rho u) + \p_x (\rho u^2 +p)=0,\\
    \p_t (\rho E) + \p_x ((\rho E +p)u)=0,
  \end{cases}
\end{equation}
with the closure pressure law
\begin{equation}
  \label{eq:HRM_NPT_pressure}
  p=p(1/\rho,e,\varphi_l,\varphi_g,\alpha_l, z_l,z_g).
\end{equation}

One should should add un entropy criterion to the model. With
$\rho=1/\tau$ the concave function
$\sigma(\tau,e,\varphi_l,\varphi_g,\alpha_l,z_l,z_g)$ defined in
\eqref{eq:intensive_entropy} would be an entropy function if it
satisfies the first order PDE
\begin{equation}
\label{eq:entrop_relat_partial}
  \p_\tau \sigma -p (1/\tau,e,\varphi_l,\varphi_g,\alpha_l,z_l,z_g)
  \p_e \sigma=0.
\end{equation}
Setting $T=1/\p_e \sigma$, one recovers the relation
\begin{equation*}
  Td\sigma = de +pd\tau+\sum_{k=l,g}\p_{\varphi_k} s d\varphi_k +
  \p_{\alpha_l}s d\alpha_l + \sum_{k=l,g}\p_{z_k} s dz_k .
\end{equation*}

Weak solutions of \eqref{eq:HRM_NPT}-\eqref{eq:HRM_NPT_pressure}
satisfy
\begin{equation}
  \label{eq:HRM_NPT_entrop_ineq}
  \p_t (\rho\sigma) + \p_x (\rho u \sigma) \geq 0,
\end{equation}
which becomes an equality as regular solutions are concerned. The
concavity of $\sigma$ with respect to
$(\tau,e,\varphi_l,\varphi_g,\alpha_l,z_l,z_g)$
is equivalent to the convexity of $H=-\rho \sigma$ with respect to the
conservative variables $(\rho,\rho u, \rho E, \varphi_l \rho,
\varphi_g \rho, z_l \rho, z_g \rho, \alpha_l \rho)$, following
\cite{croisille91, GodRav91}. Hence $H=-\rho \sigma$ is a Lax entropy
for \eqref{eq:HRM_NPT}.

In order to bring the system to thermodynamical equilibrium described
in Proposition \ref{prop:intensive_s}-\eqref{eq:intensive_s_NPT}, a
source term has to be added to the fractions equations
\begin{equation*}
  \p_t Y + u \p_x Y=Q.
\end{equation*}
As relaxation towards the equilibrium is infinitely fast, one recovers
the equilibrium fractions which satisfy
\begin{equation}
  \label{eq:HRM_Yeq_NPT}
  \begin{aligned}
    Y &= Y_{eq}^{NPT}(\tau,e,\varphi_l,\varphi_g)\\
    &= \argmax_{(\alpha_l,z_l,z_g)} \sigma
    (\tau,e,\varphi_l,\varphi_g,\alpha_l,z_l,z_g).
  \end{aligned}
\end{equation}
As a result the equilibrium pressure law is
\begin{equation}
  \label{eq:HRM_peq_NPT}
  \begin{aligned}
    p_{eq}^{NPT}(\tau,e,\varphi_l,\varphi_g) 
    :=
    p(\tau,e,\varphi_l,\varphi_g,
    Y_{eq}^{NPT}(\tau,e,\varphi_l,\varphi_g)),
  \end{aligned}
\end{equation}
defined by the Dalton's law \eqref{eq:p_T_eq_mixture}.
Following \cite{CoquelPerthame98, BH05, Hurisse17}, a natural source
term is
\begin{equation}
  \label{eq:Q_NPT}
  Q=\lambda(Y_{eq}^{NPT}(\tau,e,\varphi_l,\varphi_g) - Y)
\end{equation}
where the parameter $\lambda$ goes to $+\infty$ to achieve the
thermodynamical equilibrium.

Moreover the source term $Q$ complies with the entropy production
criterion since 
\begin{equation}
  \label{eq:HRM_dissip_entrop_NPT}
  \begin{aligned}
    \p_t \sigma + u \p_x \sigma &= \nabla_Y \sigma \cdot (\p_t Y + u
    \p_x Y)\\
    &= \lambda \nabla_Y \sigma \cdot
    (Y_{eq}^{NPT}(\tau,e,\varphi_l,\varphi_g) - Y)\\
    &\geq \lambda (\sigma(\tau,e,\varphi_l,\varphi_g,Y_{eq}^{NPT})-
    \sigma  (\tau,e,\varphi_l,\varphi_g,\alpha_l,z_l,z_g))\\
    &\geq 0,
  \end{aligned}
\end{equation}
by concavity of the entropy $\sigma$.

The drawback of the source term \eqref{eq:Q_NPT} is that the
relaxation parameter $\lambda$ is identical for all the fractions. Hence
the relaxation times towards the mechanical and thermal equilibrium
are the same, which has no particular physical meaning.
An alternative which guarantees the entropy production is 
\begin{equation}
  \label{eq:Q_NPT_2}
  \begin{aligned}
    Q&=\nabla_Y \sigma (\tau,e,\varphi_l,\varphi_g,Y),\\
    & =
    \begin{pmatrix}
      \tau \left( \dfrac{p_l}{T_l} - \left( \dfrac{p_g}{T_g}+
          \dfrac{p_v}{T_v}\right)\right) \\
      e \left( \dfrac{1}{T_l} - \dfrac{1}{T_v}\right)\\
      e \left( \dfrac{1}{T_g} - \dfrac{1}{T_v}\right)
    \end{pmatrix},
  \end{aligned}
\end{equation}
since $\p_t \sigma + u \p_x \sigma = |\nabla_Y\sigma|^2\geq 0$.
This choice of source term enables to use different
relaxation scales for mechanical and thermal equilibria.

\subsection{HRM model with phase transition}
  \label{sec:hrm-NPT}

Following the same methodology explained in  Section \ref{sec:hrm-PT},
one obtains the HRM model taking into account phase transition between
the liquid and the vapor. It reads
\begin{equation}
  \label{eq:HRM_PT}
  \begin{cases}
    \p_t Y +   u\p_x Y=Q, \\
    \p_t (\varphi_g \rho) +   \p_x (\varphi_g\rho u) =0, \\
    \p_t \rho + \p_x (\rho u) =0,\\
    \p_t (\rho u) + \p_x (\rho u^2 +p)=0,\\
    \p_t (\rho E) + \p_x ((\rho E +p)u)=0,
  \end{cases}
\end{equation}
where the fraction vector is $Y=(\varphi_l,\alpha_l,z_l,z_g)$ and
the closure pressure law
\begin{equation}
  \label{eq:HRM_NPT_pressure}
  p=p(1/\rho,e,\varphi_g,\varphi_l,\alpha_l, z_l,z_g).
\end{equation}
Again the entropy
$\sigma(\tau,e,\varphi_g,\varphi_l,\alpha_l,z_l,z_g)$ defined in
\eqref{eq:intensive_entropy} satisfy the entropy inequality
\eqref{eq:HRM_NPT_entrop_ineq} as soon as it complies with
\eqref{eq:entrop_relat_partial}. 
The source term $Q$ has to be chosen to recover the thermodynamical
equilibrium described by the fractions
\begin{equation}
  \label{eq:HRM_Yeq_PT}
  \begin{aligned}
    Y &= Y_{eq}^{PT}(\tau,e,\varphi_g)\\
    &= \argmax_{(\varphi_l,\alpha_l,z_l,z_g)} \sigma
    (\tau,e,\varphi_l,\varphi_g,\alpha_l,z_l,z_g),
  \end{aligned}
\end{equation}
leading to the equilibrium pressure law
\begin{equation}
  \label{eq:HRM_peq_PT}
  \begin{aligned}
    p_{eq}^{PT}(\tau,e,\varphi_g)
    :=
    p(\tau,e,\varphi_g,
    Y_{eq}^{PT}(\tau,e,\varphi_g)).
  \end{aligned}
\end{equation}
Again the source term $Q$ could be either
\begin{equation}
  \label{eq:Q_PT}
  Q=\lambda(Y_{eq}^{PT}(\tau,e,\varphi_g)-Y),
\end{equation}
or
\begin{equation}
  \label{eq:Q_PT_2}
  \begin{aligned}
    Q &=\lambda
    \nabla_Y\sigma(\tau,e,\varphi_g,Y_{eq}^{PT}(\tau,e,\varphi_g))\\
    & =
    \begin{pmatrix}
      s_l-\tau_l\dfrac{p_l}{T_l}-\dfrac{e_l}{T_l}
      -s_v+\tau_v\dfrac{p_v}{T_v}+\dfrac{e_v}{T_v}\\
      \tau \left( \dfrac{p_l}{T_l} - \left( \dfrac{p_g}{T_g}+
          \dfrac{p_v}{T_v}\right)\right) \\
      e \left( \dfrac{1}{T_l} - \dfrac{1}{T_v}\right)\\
      e \left( \dfrac{1}{T_g} - \dfrac{1}{T_v}\right).
    \end{pmatrix}
  \end{aligned}
\end{equation}
Using the characterization \eqref{eq:mu} of the chemical potential,
the first component of $Q$ boils down to
\begin{equation}
  \label{eq:Q_PT_2_1}
  \p_{\varphi_l} \sigma = -\dfrac{\mu_l}{T_l}+
  \dfrac{\mu_v}{T_v},
\end{equation}
which reflects the mass transfer between the liquid and its vapor.\\

\textbf{Acknowledgment.} The author would like to thank the Centre
Henri Lebesgue ANR-11-LABX-0020-01 for creating an attractive
mathematical environment.
\bibliographystyle{plain}
\bibliography{3phases}

\end{document}